\documentclass[12pt,thmsa, leqno]{amsart}

\let\oldtocsection=\tocsection

\let\oldtocsubsection=\tocsubsection

\renewcommand{\tocsection}[2]{\hspace{0em}\oldtocsection{#1}{#2}}
\renewcommand{\tocsubsection}[2]{\hspace{2em}\oldtocsubsection{#1}{#2}}

\usepackage{setspace}

\usepackage{amssymb, euscript,  mathrsfs}
\usepackage[dvips]{graphics}
\usepackage[title]{appendix}

\input{amssym.def}

\input{amssym.tex}

\usepackage{lineno}

\usepackage{tikz}
\usepackage[utf8]{inputenc}

\usetikzlibrary{matrix,arrows,decorations.pathmorphing}

\usepackage[all,cmtip]{xy}

\usepackage{amsmath}
\let\oldAA\AA
\renewcommand{\AA}{\text{\normalfont\oldAA}}

\def\part{\partial}
\def\intl{\int\limits}
\def\b{\beta}

\def\Gam{\Gamma}
\def\Om{\Omega}
\def\a{\alpha}
\def\om{\omega}

\def\suml{\sum\limits}

\def\del{\delta}
\def\vp{\varphi}

\def\g{\gamma}
\def\gam{\gamma}

\def\sig{\sigma}
\def\lam{\lambda}

\def\const{{\hbox{\rm const}}}
\def\cos{{\hbox{\rm cos}}}

\def\bbr{{\Bbb R}}

\newtheorem{theorem}{Theorem}[section]
\newtheorem{lemma}[theorem]{Lemma}

\theoremstyle{definition}
\newtheorem{definition}[theorem]{Definition}
\newtheorem{example}[theorem]{Example}

\theoremstyle{remark}
\newtheorem{remark}[theorem]{Remark}

\numberwithin{equation}{section}

\theoremstyle{corollary}
\newtheorem{corollary}[theorem]{Corollary}

\numberwithin{equation}{section}

\newcommand{\be}{\begin{equation}}
\newcommand{\ee}{\end{equation}}

\newcommand{\bea}{\begin{eqnarray}}
\newcommand{\eea}{\end{eqnarray}}
\newcommand{\Bea}{\begin{eqnarray*}}
\newcommand{\Eea}{\end{eqnarray*}}

\def\sideremark#1{\ifvmode\leavevmode\fi\vadjust{\vbox to0pt{\vss
 \hbox to 0pt{\hskip\hsize\hskip1em
\vbox{\hsize2cm\tiny\raggedright\pretolerance10000
 \noindent #1\hfill}\hss}\vbox to8pt{\vfil}\vss}}}%

\begin{document}

\title[Shifted Radon Transform]
{On the Injectivity of the Shifted Funk-Radon Transform and Related Harmonic Analysis}



\author{ B. Rubin}

\address{Department of Mathematics, Louisiana State University, Baton Rouge,
Louisiana 70803, USA}
\email{borisr@lsu.edu}


\dedicatory{To the memory of Professor Lawrence Zalcman}



\begin{abstract}
Necessary and sufficient conditions are obtained for injectivity of the
shifted Funk-Radon transform associated with $k$-dimensional totally geodesic submanifolds of the unit sphere $S^n$ in $\mathbb{R}^{n+1}$.
This result generalizes the well known statement for the spherical means on $S^n$ and is formulated
in terms of  zeros of   Jacobi polynomials. The relevant harmonic analysis is developed,  including a new concept
of induced Stiefel (or Grassmannian) harmonics, the Funk-Hecke type theorems, addition formula, and multipliers. Some perspectives and conjectures are discussed.
 \end{abstract}

\maketitle

\section{Introduction}

Let $X$ be an $n$-dimensional constant curvature space,   $\Xi$ be the set of all $k$-dimensional totally geodesic submanifolds of $X$, $1\le k \le n-1$,  \cite{H11}.
Consider  the Radon type transform
\be\label{Radon}
(R_t f)(\xi)=\!\!\intl_{d(x,\xi)=t}\!\! f(x) dm(x), \qquad x\in X, \quad \xi \in \Xi, \quad t>0,\ee
where  $d (\cdot, \cdot)$ stands for the  geodesic distance on $X$ and  $dm(x)$ is the relevant canonical measure.

{\bf Question.} {\it Suppose that $t>0$ is fixed. How does the injectivity of $R_t$ depend on the values of $t$ and the class of functions $f$?}

 For the  spherical means on $X$, formally corresponding to $k=0$, this problem  was studied by  Berenstein and Zalcman \cite[Section 6]{BZ76}. It
  falls into the scope of the wide class of  Pompeiu's problems. There is an extensive literature related to numerous aspects of the spherical means  and the Pompeiu problem in general; see, e.g.,
 \cite{AQ, Z80, Z92, Z01} and  references therein.

The operator $R_t$ and its dual
\be\label{Radondu}(R^*_t \vp)(x)=\!\!\intl_{d(x,\xi)=t}\!\! \vp (\xi) d\mu (\xi)\ee
are well known  in integral geometry \cite{H11, Ru13b}. Following   Rouvi\`{e}re \cite [p. 19]{Rou}, we call $R_t f$ and $R^*_t \vp$ the   {\it shifted Radon transform} and the {\it shifted dual Radon transform}, respectively.
 The terminology is motivated by the fact that the limiting case $t=0$ yields the well known totally geodesic Radon transform and its dual \cite{H11}.

In the present article we are focusing on the case when $X$ is the unit sphere $S^n$ in  $\bbr^{n+1}$ and call  $R_t f$  the  {\it  shifted Funk-Radon transform}, because this name  is more precise.
 The  functions on $\Xi$ can
be thought of as the  functions on the Grassmann manifold $G_{n+1, k+1}$ of $(k+1)$-dimensional linear subspaces of $\bbr^{n+1}$. Alternatively,  they can be interpreted as right
$O(n-k)$-invariant functions on the Stiefel manifold $V_{n+1, n-k}$ of orthonormal $(n-k)$-frames in $\bbr^{n+1}$.

\vskip 0.2 truecm

\noindent {\bf Main Results.} We invoke the Jacobi polynomials $P_{j/2}^{(\sigma,\rho)}$ \cite {Er} with $j$ even.

\begin{theorem} \label{inje} Let $1\le k \le n-1$, $\sigma=(n-k)/2 -1$, $\rho=(k-1)/2$.

\vskip 0.2 truecm

\noindent {\rm (i)} The  operator  $R_t$   with fixed $t\in (0,\pi/2)$  is injective on $L^1_{even} (S^n)$ if and only if $P^{(\sigma,\rho)}_{j/2} (\cos \,2 t)\neq 0$  for all $j\in \{0,2,4, \ldots\}$.

\vskip 0.2 truecm

\noindent {\rm (ii)} More generally, given a positive integer $\ell$, let  $f\in L^1_{even} (S^n)$ and suppose that $R_{t_i} f=0$ a.e. for all $t_i\in (0,\pi/2)$;  $\;i=1,2,\ldots, \ell$.  If the equations
\be\label {sol}
 P^{(\sigma,\rho)}_{j/2} (\cos \,2 t_i)= 0, \qquad  i=1,2,\ldots, \ell, \ee
 have no common  solution for  $j\!\in \!\{0,2, 4,\ldots\}$, then $f\!=\!0$ a.e. on $S^n$. If these equations have a common solution, say, $j=j_0$, then $R_{t_i} Y_{j_0}=0$ for all $\;i=1,2,\ldots, \ell$ and all spherical harmonics $Y_{j_0}$ of degree $j_0$.
\end{theorem}

This theorem agrees with the known  case $k=n-1$ (cf. \cite[Theorem 8]{BZ76}), when the Jacobi polynomial $P_{j/2}^{(\sigma,\rho)}$ can be written as   the  Gegenbauer polynomial $C_j^{(n-1)/2}$  with transformed argument; use, e.g.,  \cite[formula 10.9 (21)]{Er}.

\begin{corollary}\label{Ungar}  The set of all $t\in (0,\pi/2)$, for which $R_t$ is non-injective on $L^1_{even} (S^n)$,  is everywhere dense in $(0,\pi/2)$ and so is the set  for which it does.\end{corollary}

This statement follows from the density property of zeros  of orthogonal polynomials (see, e.g., \cite[Theorem 6.1.1]{Sz}) and the fact that the set of all such  zeros  is countable.
It mimics the celebrated  Ungar's freak theorem for spherical caps in $S^2$; see also   Schneider \cite{Schn1, Schn2} and    Berenstein
and  Zalcman \cite{BZ76} regarding the similar statement for hyperplane  sections of $S^n$,  $n\ge 2$.

\begin {remark} Theorem \ref{inje} gives no answer about  injectivity of $R_t$ for particular values of $t \in (0,\pi/2)$, say, $t=\pi/5$ or  $t =\pi/6$. It only reformulates the problem in a different  language.
However, this reformulation is very important. For instance, it allows one to invoke the tools of number theory and  asymptotic properties of Jacobi polynomials for further invesstigation. Some results in this direction for $k=n-1$ and associated Legendre functions can be found  in \cite{Ru00a}, \cite[Section 5.5]{Ru15}.
\end{remark}

To prove Theorem \ref{inje}, we introduce a new concept of   induced Stiefel harmonics  on $V_{n+1, n-k}$.
 These harmonics are right $O(n-k)$-invariant, constitute an orthonormal system, and can  be regarded as harmonics on the Grassmann manifold $G_{n+1, k+1}$. They
are generated by the usual spherical harmonics on $S^n$.  We prove the  addition formula for such harmonics and
  establish  new Funk-Hecke type theorems for $O(n+1)$-intertwining
operators, which connect functions on $S^n$ with functions on $V_{n+1, n-k}$ (or $G_{n+1, k+1}$). These theorems
  provide explicit formulas for the relevant  Fourier-type multipliers. The Jacobi polynomials are the main ingredients of these formulas.

 The developed harmonic analysis is applicable not only to the operators
 (\ref{Radon}) and (\ref{Radondu}) but also to the  Funk-Radon transforms (the case $t=0$) and to the more general analytic
 families of generalized cosine transforms  in integral geometry \cite {Ru02b, Ru08}; see examples in Section \ref{Examples}.

   Section 2 contains preliminaries. In Section 3 we prove the main results. More comments can be found in Conclusion, also containing   some thoughts about
    possible developments in the future.

\section{ Preliminaries}

\subsection{ Notation}
${}$

\noindent In the following, $\bbr^{n+1}$, $n\ge 2$, is the real $(n+1)$-dimensional Euclidean space  with the coordinate unit vectors $e_1, \ldots, e_{n+1}$;
$S^n \subset \bbr^{n+1}$ is the $n$-dimensional unit sphere with the area $\sigma_n = 2 \pi^{(n+1)/2}\big/\Gamma\left((n+1)/2\right)$. The points in $\bbr^{n+1}$ will be identified with the relevant column vectors.

For  $x\in S^{n}$, we write $dx$ for the surface area measure on $S^n$ and set $d_*x = \sigma_n^{-1} dx$ for the corresponding normalized
  measure.

Given an integer $k$,  $1\le k \le n-1$,   let
$V_{n+1, n-k}$ be the Stiefel manifold  of orthonormal  $(n-k)$-frames in $\bbr^{n+1}$. Every element $v\in V_{n+1, n-k}$ is an  $(n+1)\times (n-k)$  matrix  satisfying
$v^T v=I_{n-k}$, where $v^T$ is the transpose of $v$ and $I_{n-k}$ is the identity $(n-k)\times (n-k)$ matrix.
 We equip $V_{n+1, n-k}$ with the standard probability measure $d_*v$,  which is left $O(n+1)$-invariant and right $O(n-k)$-invariant. Given $v\in V_{n+1, n-k}$, we denote by $v^\perp$ the
 $(k+1)$-dimensional linear subspace of $\bbr^{n+1}$ orthogonal to $v$. The  Grassmann manifold  of all such subspaces will be denoted by  $G_{n+1, k+1}$.
 If $x\in S^n$, $v\in V_{n+1, n-k}$, and  $\{ v \}$ is the $(n-k)$-plane  spanned by $v$, then  $|x^T v|$
 is the length of the orthogonal  projection of $x$ onto  $\{ v \}$. The notation $x\cdot y=x_1 y_1 +\cdots + x_{n+1} y_{n+1}$ for the vectors $x,y \in \bbr^{n+1}$  is standard.

Let $\Xi$ be the set of all $k$-dimensional totally geodesic submanifolds $\xi$ of $S^n$ ($k$-geodesics, for short) equipped with the canonical  $O(n+1)$-invariant probability measure $d_*\xi$.
Every  $\xi  \in \Xi$ has the form  $\xi = S^{n} \cap v^\perp$ for some $v\in V_{n+1, n-k}$.

The group $O(n+1)$ of orthogonal transformations of $\bbr^{n+1}$ and all subgroups of $O(n+1)$ will be equipped with the corresponding Haar measure of total mass one.

In the following, $v_0 =\left[\begin{array}{c} 0 \\ I_{n-k}
\end{array} \right]\in  V_{n+1, n-k}$ denotes the coordinate frame. If $v \in  V_{n+1, n-k}$ and  $r_v \in  O(n+1)$  maps $v_0$ to $v$, we set $f_v(x) = f(r_vx)$.
 Similarly, if $x \in S^n$  and $r_x  \in  O(n+1)$  maps $e_{n+1}$ to $x$, we denote $\vp_x (v) =\vp(r_x v)$.

We say that  an integral under consideration
 exists in the Lebesgue sense if it is finite when the integrand is replaced by its absolute value.

\subsection{Bispherical means}

In this section we give precise meaning to the shifted Funk-Radon transform $R_t$ on $S^n$ and its dual  $R_t^*$. The operator $R_t$
will be realized as a certain bispherical mean associated with bispherical coordinates  in $S^n$. We also recall some known facts about spherical harmonics and their representation in bispherical coordinates.

Let
\[\bbr^{n+1} = \bbr^{k+1} \times \bbr^{n-k},\qquad 1\le k \le n-1.\]
\be\label {lazx}  \bbr^{k+1} = \bbr e_1 \oplus \ldots \oplus \bbr e_{k+1}, \qquad \bbr^{n-k} = \bbr e_{k+2}\oplus \ldots \oplus \bbr e_{n+1}.\ee
 Every point $x\in S^n$ can be represented  as
\be\label {2.9} x= \eta \sin \theta + \zeta \cos \,\theta, \ee
where
\[\eta\! \in\! S^k\! \subset\! \bbr^{k+1}, \qquad \zeta \! \in\! S^{n-k-1} \!\subset\! \bbr^{n-k}, \qquad 0 \!\le\!\theta \!\le\! \pi/2,\]
\[ dx=\sin^k\theta \; \cos^{n-k-1}\theta \;  d\eta d\zeta d\theta, \]
 $dx, d\eta, d\zeta$ being the corresponding non-normalized surface area measures; see, e.g., \cite [pp. 12, 22]{VK}. We
 recall that the relevant normalized measures are denoted by $d_* x, d_* \eta, d_* \zeta$.

  The variables $(\eta, \zeta, \theta)$ are called the {\it bispherical coordinates} of $x$.

Let $v_0 =\left[\begin{array}{c} 0 \\ I_{n-k}
\end{array} \right]\in  V_{n+1, n-k}$ be the coordinate frame, $r_v$ be a rotation mapping $v_0$ to $v \in  V_{n+1, n-k}$, $f_v(x) = f(r_vx)$.
Consider the integral
\be\label {3.4}  (M_{{\rm cos}\,\theta} f)(v) = \intl_{S^{n-k-1}} d_*\zeta \intl_{S^k} f_v (\eta \sin \theta + \zeta \cos \,\theta) d_*\eta\ee
or (set $\tau= \cos \,\theta$)
\be\label {3.4z}  (M_\tau f)(v) = \intl_{S^{n-k-1}} d_*\zeta \intl_{S^k} f_v (\eta \sqrt{1-\tau^2} + \zeta \tau) d_*\eta.\ee
We call $ (M_\tau f)(v)$ the {\it  bispherical mean of $f$ in the direction of  $ v \in V_{n+1, n-k}$ at the level $\tau$.} One can also write
\be\label{3.23} (M_\tau f)(v) =  \intl_{| x^T v| = \tau} f(x) d_\tau x \ee
(see Notation), where $d_\tau x$ stands for the corresponding probability measure.

The integral  (\ref{3.4}) gives precise meaning to the shifted Funk-Radon transform  (\ref{Radon}). Specifically,
\be\label {cdf}  (M_{{\rm cos}\,\theta} f)(v) =  \!\intl_{d(x, \xi)=t} \!\! f(x)\, dm(x)\equiv(R_t f)(\xi),\ee
\[
 t= \frac{\pi}{2}- \theta,\qquad \xi=S^n \cap v^\perp\in \Xi.\]
By (\ref{2.9}), for any  $v \in  V_{n+1, n-k}$ we have
\be\label {3.5} \intl_{S^n} f(x) d_*x = \frac{\sigma_{n-k-1} \sigma_k}{\sigma_n} \intl^{\pi/2}_0  (M_{{\rm cos}\,\theta} f)(v) \sin^k\theta \,\cos^{n-k-1} \theta \,d \theta. \ee

To define the dual  of $(M_\tau f)(v)$, we first write
\bea \label {3.6} \qquad \quad (M_\tau f)(v)\!\! &=& \!\!\!\intl_{SO(n-k)} \!\!\! d\a\!\!\intl_{SO(k+1)} \!\!\!f_v(\a e_{n+1} \cos \,\theta + \b e_{k+1}\sin \theta)\,d\b\qquad \\
\label {3.7}\!\!&=&\!\!\intl_{K'} f_v(ge_{n+1}) d\g, \qquad g = \g g_{k+1, n+1} (\theta)\delta, \eea
 where $K'=SO(n-k)\times SO(k+1), \; \; \delta \in K = SO (n)$, $\tau = \cos\,\theta$, and
$g_{k+1, n+1} (\theta)$ is a rotation in the plane $(e_{k+1}, e_{n+1})$ with the matrix
\[
\left[\begin{array}{cc}\cos\theta & \sin\theta \\
                              -\sin\theta & \cos \theta
\end{array} \right].\]
Clearly, $g_{k+1, n+1} (\theta) e_{n+1} = e_{n+1}\cos \,\theta + e_{k+1} \sin \theta$. We define
\be\label {3.8}  (M^*_\tau \vp)(x) = \intl_{K'\times K} \vp_x (g^{-1}    v_0) d \gamma d \delta,  \ee
where $g $ has the same meaning as in (\ref{3.7}), $\vp_x (v) =\vp(r_x v)$, $r_x$ is a rotations which maps $e_{n+1}$ to $x$.

If $\vp $ is a right $O(n-k)$-invariant function on $V_{n+1, n-k}$, which is interpreted as a function of $\xi \in \Xi$, then, abusing notation, we can write  (\ref{3.8}) as
\be\label {cdfa}
(M^*_{{\rm cos}\,\theta}  \vp)(x) =  \!\intl_{d(x, \xi)=t} \!\! \vp (\xi) d\mu (\xi)\equiv(R^*_t \vp)(x), \qquad t= \frac{\pi}{2}- \theta.\ee
 This integral gives precise meaning to the shifted dual Funk-Radon transform  (\ref{Radondu}).

We recall that  $(R_t f)(\xi)$  and $(R^*_t \vp)(x)$ with $t=0$ are the usual
 Funk-Radon transforms  \cite{H11, Rou, Ru02b}, so that
\be\label {cdfae}
(M_0 f)(v)\!= \!(Rf)(S^n \cap v^\perp)\!=\!(Rf)(\xi), \quad (M^*_0 \vp)(x)\!=\! (R^* \vp)(x), \ee
if we identify $O(n-k)$-invariant functions $\vp$ on  $V_{n+1, n-k}$ with functions on $\Xi$.

\begin{lemma}\label {Lemma 3.1}  For any $\tau \in [0, 1]$,
\be\label {3.10}
\intl_{V_{n+1, n-k}} (M_\tau f)(v) \vp (v) d_*v = \intl_{S^n} f(x) (M^*_\tau \vp)(x) d_*x, \ee
provided that either side of this equality exists in the Lebesgue sense.
\end{lemma}
\begin{proof} This duality statement is well know in the general double fibration context \cite{H11}. For the sake of completeness, we present its proof in the Stiefel terms.
Let $G=SO (n+1)$, $\tau = \cos \,\theta$. By (\ref{3.7}),
\bea I&=&\intl_{V_{n+1, n-k}} (M_\tau f)(v) \vp (v) d_*v = \intl_G (M_{\tau} f)(gv_0) \vp (gv_0) dg  \nonumber\\
&=&\intl_{K'} d\gamma \intl_G \vp (gv_0) f (g\gamma g_{k+1, n+1} (\theta) \delta e_{n+1}) dg. \nonumber\eea
Now we change the notation  $g\g g_{k+1, n+1} (\theta)\delta \to g$, then integrate in $\delta \in K$, and change the order of integration. This gives
\bea I&=&\intl_G f(ge_{n+1}) dg \intl_{K'\times K} \vp(g\delta^{-1} [g_{k+1, n+1} (\theta)]^{-1} \gamma^{-1} v_0)d\gamma d\delta  \nonumber\\
&=&\intl_G f(ge_{n+1}) (M^*_{\tau} \vp) (g e_{n+1}) dg,  \nonumber\eea
 which implies (\ref {3.10}).
 \end{proof}

\subsection {Spherical harmonics in bispherical coordinates}
${}$

\vskip 0.2 truecm

\noindent {\bf 1.}
Let ${\mathscr Y}=\{ Y^n_{j, \lambda} (x)\}$ be an orthonormal basis  of spherical harmonics in $L^2 (S^n)$.
Here $j \in \{ 0, 1, \dots\}$,  $\; \lambda \in \{ 1, 2, \dots, d_n(j)\}$;
\be\label {2.8} d_n(j) =  (n+2j-1) \, \frac{\Gam (n+j-1)}{\Gam (j+1)\,\Gam (n) }\ee
 is the dimension of the subspace of spherical harmonics of degree $j$. Thus
 \be\label {tagw}
 \intl_{S^n} Y^n_{j, \lam} (x)  Y^n_{j', \lam'} (x)\, d_*x =\left\{ \begin{array} {ll} 1
 &\;\;\hbox{\it if} \; \; j = j' \; \; \hbox{\it and} \; \; \lam = \lam',\\ 0 &\; \; \hbox{\it otherwise}\\
\end{array}
\right.\ee
(it is important to keep in mind that normalization of the spherical harmonics throughout the paper is understood with respect to the probability measure $d_* x$, not with respect to the surface area measure $dx$).

 If $Y_j$ is  a spherical harmonic  of degree $j$,   $\Om (t) (1\!-\!t^2)^{n/2-1} \!\in L^1 (-1,1)$, then,  by the  Funk-Hecke theorem,
\be\label {FH} \intl_{S^n}\! \Om (x\cdot y) Y_j (x) d_*x= \om_j Y_j (y), \ee
where \[ \om_j=  \frac{\sig_{n-1}}{\sig_n} \! \intl^1_{-1}\! P_{j} (t) \Omega (t) (1-t^2)^{n/2-1} dt,\]
\be\label {als} P_{j} (t)= \displaystyle{\frac{j!\, (n-2)!}{ (j+n-2)!}}\, C^{(n-1)/2}_{j} (t), \ee
 $C^{(n-1)/2}_{j} (t)$ being the Gegenbauer polynomial. The polynomials (\ref{als}) are called the {\it spherical  polynomials} (other names are also  known) and enjoy the following properties:
\be\label {shd1}  P_j(1)=1;\ee
\be\label {shap31} \intl_{-1}^1 P_j(t)\,
P_{j'}(t)\, (1-t^2)^{n/2-1} dt=
\begin{cases}
0,  & \text {\rm if} \; j\neq j',\\
\displaystyle {\frac{\sig_{n}}{d_n (j)\,\sig_{n-1} },}  &\text {\rm if} \; j=j';
\end{cases}
\ee
\be \label {hva34} \sum\limits_{\lam =1}^{d_n(j)} Y^n_{j, \lambda}
(x)\, Y^n_{j, \lambda}  (y)= d_n(j)\,P_j(x\cdot y).\ee

The  reader is referred to \cite[Section A.6]{Ru15}, where these statements are proved in slightly different notation.

\vskip 0.2 truecm

{\bf 2.}
We will need representation of spherical harmonics in the bispherical coordinates
\[ x= \eta \sin \theta + \zeta \cos \theta,\]
\[\eta\! \in\! S^k\! \subset\! \bbr^{k+1}, \qquad \zeta \! \in\! S^{n-k-1} \!\subset\! \bbr^{n-k}, \qquad 0 \!\le\!\theta \!\le\! \pi/2;\]
cf. (\ref{2.9}).
  Let $P_m^{(\rho,\sigma)} (t)$, $m \in \{0,1,2, \ldots \}$, be the  Jacobi polynomials; $ \rho, \sig > -1$. The corresponding normalized polynomials are defined by
\be\label {2.10} R^{(\rho,\sigma)}_m(t) = \, \frac{P^{(\rho, \sigma)}_m (t)}{P_m^{(\rho,\sigma)} (1)}, \qquad P^{(\rho,\sigma)}_m (1)  =\frac{ \Gamma(m+\rho + 1)}{m!\, \Gamma (\rho + 1)}. \ee
We recall  that
\bea\label {2.11}  &&\intl^1_{-1} [R^{(\rho,\sigma)}_m (t) ]^2 (1-t)^\rho (1+t)^\sigma dt\\
&&=\frac{2^{\rho+\sigma+1}m! \;
\Gamma^2(\rho+1) \Gamma (m+\sigma+1)}{(2m+\rho + \sigma+1) \Gamma (m+\rho+1) \Gamma (m+\rho +\sigma+1)}\nonumber\eea
and
\be\label {ced}
\intl^1_{-1} R^{(\rho,\sigma)}_\ell (t) \,R^{(\rho,\sigma)}_m (t)  (1-t)^\rho (1+t)^\sigma dt=0, \qquad \ell\neq m;\ee
cf.  \cite[pp. 300, 301]{AAR}.  Note also  (see, e.g., \cite[formula 10.9 (21)]{Er})  that
\be\label{kin}
R^{(n/2 -1, -1/2)}_{m} (2t^2 -1) =P_{2m} (t), \ee
$P_{2m}$ being the spherical polynomial (\ref{als}) of degree $j=2m$.

Let $\{ Y^k_{r, \mu}(\eta)\} $ and $\{ Y^{n-k-1}_{s,\nu} (\zeta)\} $ be orthonormal bases of spherical harmonics in $L^2 (S^k)$ and $L^2 (S^{n-k-1})$, respectively. Here
\[
r,s=0, 1, \dots;\qquad  \mu = 1, \dots, d_{k+1} (r); \qquad  \nu = 1, \dots, d_{n-k} (s);\]
cf. (\ref {2.8}). We set
$$\rho = r+(k-1)/2, \qquad \sigma = s+ (n-k)/2 -1,  $$
and consider the collection of functions
\be\label {2.12} U^j_M (x) = \varkappa_M \,Y^k_{r,\mu} (\eta)\, Y^{n-k-1}_{s,\nu} (\zeta)\, \sin^r\theta\, \cos^s\theta \,R^{(\rho, \sigma)}_m (\cos 2\theta),\ee
indexed by $M = (r, \mu; s, \nu; m)$ with $j = 2m+r+s$ and
\be\label {2.13} \varkappa^2_M=\frac{2\sigma_n\,(2m+\rho +\sigma +1)\; \Gamma\; (m+\rho+1) \; \Gamma \; (m+\rho +\sigma + 1)}{\sigma_{n-k-1} \sigma_k\,m!\; \Gamma \, (m+\sigma + 1)\; \Gamma^2\, (\rho + 1)}. \ee
 Each $U^j_M (x)$ is a spherical harmonic of degree $j$. We denote by  ${\mathscr U}$
  the collection   of all harmonics (\ref{2.12}). One can show \cite  [pp. 208 - 211] {VK}
  that ${\mathscr U}$ is an orthonormal basis in $L^2(S^n)$.

For convenience of the reader, let us check, for instance, that
\be\label {loi} ||U^j_M ||_{L^2 (S^n)}=1.\ee
 Passing to bi-spherical coordinates (\ref{2.9}) and taking into account normalization, we have
\bea
 ||U^j_M ||^2_{L^2 (S^{n})}&=& \intl_{S^n} |U^j_M (x)|^2\, d_*x \nonumber\\
&=&\varkappa^2_M \, \frac{\sigma_{n-k-1} \sigma_k}{\sigma_n}  \intl_{S^k} |Y^k_{r,\mu} (\eta)|^2  d_*\eta  \intl_{S^{n-k-1}} |Y^{n-k-1}_{s,\nu} (\zeta)|^2 d_*\zeta \nonumber\\
&\times& \intl_0^{\pi/2} \sin^{2r+k} \theta\, \cos^{2s+ n-k-1}\theta \,|R^{(\rho, \sigma)}_m (\cos 2\theta)|^2  d\theta\nonumber\\
&=&\varkappa^2_M\, \frac{\sigma_{n-k-1} \sigma_k}{\sigma_n}\, I;\nonumber\eea
\[ I=\intl_0^{\pi/2} \sin^{2r+k} \theta\, \cos^{2s+ n-k-1}\theta |R^{(\rho, \sigma)}_m (\cos 2\theta)|^2  d\theta.\]
Changing variables and using (\ref{2.11}), we obtain
\[ I= \frac{ m! \;
\Gamma^2(\rho+1) \Gamma (m+\sigma+1)}{2(2m+\rho + \sigma+1) \Gamma (m+\rho+1) \Gamma (m+\rho +\sigma+1)}.\]

Both orthonormal bases   ${\mathscr Y}=\{ Y^n_{j, \lambda} (x)\}$ and ${\mathscr U}=\{U^j_M (x)\}$  will be needed in the next sections.

\section{Intertwining Operators. Main results}

Consider a dual pair of integral operators of the form
\bea \label {3.1} (Af)(v) &=& \intl_{S^n} a(|x^T  v|) \,f(x) \,d_*x, \qquad v \in V_{n+1, n-k},\\
\label {3.2}({A^*}\vp)(x) &=& \intl_{V_{n+1, n-k}} \!\!\! a(|x^T  v|) \, \vp (v) \,d_*v, \qquad x \in S^n,\eea
which intertwine the action of the orthogonal group $O(n +1)$ on $S^n$ and $V_{n+1, n-k}$.  Here $1 \le k \le n-1$ and $a$ is a function on $[0, 1]$.

Clearly, $(Af)(v\g) = (Af)(v)$ for all $\g \in O(n-k)$, and therefore $(A f) (v)$ can be viewed as a function on the  Grassmannians or on the space of $k$-geodesics. Specifically,
\be\label {3.3}  (Af)(v) \equiv \left\{ \begin{array} {ll}  (A_1 f)(v^\perp), \quad &v^\perp \in G_{n+1, k+1},\\
(A_2 f)(\{ v \}), \quad &\{ v \} \in G_{n+1, n-k},\\
(A_3 f)(S^n \cap v^\perp), \quad & S^n \cap v^\perp \in \Xi.
\end{array} \right. \ee

\begin{lemma}\label{Lemma 3.2}  Suppose that the integrals (\ref{3.1}) and (\ref{3.2}) are absolutely convergent.  Then
\bea \label {3.11} (Af)(v) &=& c_{n, k} \intl^1_0  a(\tau) (M_\tau f)(v) \,\rho (\tau)\, d\tau,\\
 \label {3.12}
 (A^*\vp) (x) &=& c_{n, k} \intl^1_0 a(\tau) (M^*_\tau \vp)(x) \,\rho (\tau)\,d\tau, \eea
where
 \be\label {3.13} \rho (\tau) = (1-\tau^2)^{(k-1)/2} \tau^{n-k-1},\qquad   c_{n,k}=\frac{\sigma_{n-k-1} \sigma_k}{\sigma_n}.\ee
\end{lemma}
\begin{proof}  Since $(Af)(v) = \int_{S^n} a(|x^T  v_0 |) f_v (x) d_*x$, (\ref{3.11}) follows from (\ref{3.5}). Furthermore,
\bea (A^*\vp)(x)\!\!&=&\!\!\!\intl_{V_{n+1, n-k}} \!\!\! a(|e_{n+1}^T v|) \vp_x (v) d_*v  \nonumber\\
 \label {3.14} &=&\intl_G \!a(|e_{n+1}^T gv_0| ) \vp_x (gv_0) dg, \qquad G \!=\! SO (n+1).\quad \eea
Replace $g$ by $ \delta g^{-1}$, $\delta\! \in \!K=SO(n)$, and integrate in $\delta$.  Then
$$(A^*\vp)(x) = \intl_G a(|(ge_{n+1})^T v_0|) \om(g) dg, \quad \om(g) = \intl_K \vp_x (\delta g^{-1} v_0) d\delta.$$
Since $\om (gK) = \om(g)$, one can write $\om(g) = \Omega (ge_{n+1})$, where $\Omega$ is a function on $S^n = G/K$.  Hence
\bea (A^*\vp)(x) &=&  \intl_{S^n} a(|y^T v_0 | ) \Omega (y) d_*y =  (A\Omega) (v_0)  \nonumber\\
 &=&c_{n, k} \intl^1_0  a(\tau) (M_{\tau} \Omega) (v_0) \,\rho (\tau) \,d\tau. \nonumber\eea
 Setting $\tau = \cos\,\theta$ and using (\ref{3.7}), we have
\bea (M_{{\rm cos}\,\theta} \Omega)(v_0) &=& \intl_{K'} \Omega (\gamma g_{k+1, n+1} (\theta) e_{n+1}) d\g  \nonumber\\
&=&\intl_{K'} \om(\gamma g_{k+1, n+1} (\theta)) d\g.\nonumber\eea
The last integral coincides with (\ref{3.8}), and the result follows.
\end{proof}

\subsection {Norm estimates}

In the following, $G = SO (n+1)$, $\| \cdot \|_{(p)}$ and $\| \cdot \|_{p}$ denote
the $L^p$-norms of functions on $V_{n+1, n-k}$ and $S^n$, respectively.

\begin{lemma}\label{Lemma 3.3}  For all $1\le p \le \infty$ and $\tau = \cos \,\theta \in [0, 1]$,
  \be\label {3.15} \|M_{\tau} f\|_{(p)} \le  \| f\|_p, \qquad \|M^*_{\tau} \vp\|_p \le  \| \vp \|_{(p)}. \ee
 \end{lemma}
\begin{proof}   Let $F=M_{\tau} f$. By (\ref{3.7}),
\bea\| F\|_{(p)} &=& \| F(gv_0) \|_{L^p (G)}  \nonumber\\
&\le& \intl_{K'} \Big(\intl_G |f(g\g g_{k+1, n+1} (\theta) e_{n+1})|^p dg \Big)^{1/p} d\g  = \| f \|_p.\nonumber\eea
The second inequality in (\ref{3.15})  follows  by duality (\ref{3.10}).
\end{proof}

Operators (\ref{3.1}) and (\ref{3.2}) are represented as convolutions on $G$.  Specifically, setting
\[ \tilde f(g)= f(g e_{n+1}), \qquad \tilde a(g)= a(|e_{n+1}^T  g v_0|),\qquad g \in G,  \]
\[
\tilde \vp(\g)=\vp(\g v_0), \qquad  \tilde  a^* (\gam )= a (|(\gam e_{n+1})^T  v_0 |),\qquad \gam \in G, \]
 we have
 \[ (A f)(\gam v_0) \!= \!  \intl_G \! \tilde f (\g g^{-1}) \,\tilde a(g) dg,\quad
 (A^* \vp) (g e_{n+1} )  \!=  \!\intl_G  \!\tilde \vp (g \gam^{-1})\,\tilde  a^* (\gam) d\gam.\]

\begin{lemma}\label{Lemma 3.4} Let $ 1 \le p \le q \le \infty, \; \; 1 - p^{-1} + q^{-1} = r^{-1}$.  Then
 \[ \| A f\|_{(q)} \le c_{n,k}^{1/r}\, \| f \|_{p}\, \| a \|_{r, \rho}, \qquad  \|A^* \vp\|_{q} \le c_{n,k}^{1/r}\,\| \vp \|_{(p)}\, \| a \|_{r, \rho}, \]
where
 \[ \| a \|_{r, \rho} = \Big(\intl^1_0 |a(\tau) |^r\rho(\tau) d\tau\Big)^{1/r},\]
 $\rho(\tau)$ and $c_{n,k}$  being defined by  (\ref{3.13}).
\end{lemma}
\begin{proof}   The statement follows from Young's inequality on $G$ \cite[Chapter 5, Theorem 20.18]{HR} if we  notice that
\[\| \tilde f\|_{L^p (G)} =  \| f  \|_p, \qquad  \| \tilde \vp\|_{L^p (G)}=\| \vp \|_{(p)}, \]
\[ \| \tilde a \|_{L^r (G)} = \| \tilde  a^*\|_{L^r (G)} =  c_{n,k}^{1/r}\, \| a \|_{r, \rho}, \]
 \end{proof}

\subsection{The Funk-Hecke type theorems}
${}$

\noindent Below we introduce special orthonormal systems of functions on Stiefel and Grassmann manifolds.
 These systems are generated by spherical harmonics on $S^n$.  An analogue of the Funk-Hecke formula, leading to multiplier
 representation of the intertwining operators (3.1) and  (3.2), is obtained.

To start with, we evaluate bispherical means of spherical harmonics.
 By (\ref{3.2}), $(M_\tau Y_j) (v) = (M_\tau (Y_j)_v) (v_0)$. Let us
 decompose the function   $(Y_j)_v (x)= Y_j(r_vx)$ in the orthonormal basis  ${\mathscr U}=\{U^j_M(x)\}$ according to (\ref{2.12}). We obtain
\be\label {tag 4.2} Y_j(r_vx) = \suml_M \a^j_M(v) U^j_M(x), \quad \a^j_M (v) = \intl_{S^n} Y_j (r_v x) U^j_M(x) d_*x,\ee
$M=(r, \mu; s, \nu; m),\;  j = 2m+r+s$.  Thus,
\be\label {tag 4.3} (M_\tau Y_j) (v) \equiv (M_\tau (Y_j)_v) (v_0)  =\sum_M  (M_\tau U^j_M) (v_0) \,\a^j_M(v),\ee
where, by (\ref{3.4}) and (\ref{2.12}),
\bea
(M_\tau U^j_M) (v_0)&=& \varkappa_M (1-\tau^2)^{r/2} \,\tau^s \,R^{(\rho,\sigma)}_m (2\tau^2 -1) \nonumber\\
  &\times&\intl_{S^{n-k-1}} Y^{n-k-1}_{s,\nu}(\zeta) \,d_*\zeta \intl_{S^k} Y^k_{r,\mu} (\eta)\, d_*\eta.  \nonumber\eea
The last two integrals are zero, unless $s= 0$ and $r=0$.
 If  $r=0$ and $s=0$, we have
 \[\mu=1, \qquad \nu=1, \qquad M=(0,1;0,1; m),\]
 \be\label {tag 4.4} \rho=(k-1)/2, \qquad  \sigma=(n-k)/2 -1,\qquad j = 2m,\ee
and $\varkappa_M=\varkappa_j$, where
 \be\label {tag 4.9}
 \varkappa_j = \left [\frac{\sigma_n\,(2j+n-1) \; \Gamma \left(\displaystyle { j+k+1\over 2}\right)\;
 \Gamma  \left(\displaystyle { j+n-1\over 2}\right)}{\sigma_{n-k-1} \sigma_k\,\Gamma^2 \left(\displaystyle { k+1\over 2}\right)\;
 \Gamma  \left(\displaystyle { j\over 2} + 1\right) \; \Gamma \left(\displaystyle { j+n-k\over 2}\right)}\right ]^{1/2}\ee
(it is worth noting that if $k=n-1$, then (\ref{tag 4.9}) gives $\varkappa_j^2=d_n (j)$; cf. (\ref{2.8})).
Hence, for $\tau= \cos\,\theta$,
\[
(M_\tau U^j_M) (v_0)=\varkappa_j  R^{(\rho, \sigma)}_{j/2} (\cos 2\theta)= \varkappa_j R^{(\rho, \sigma)}_{j/2} (2\tau^2 -1)\]
and
\bea
 \a^j_M(v) &=&\varkappa_j\intl_{S^n} Y_j (r_v x) R^{(\rho, \sigma)}_{j/2} (2|x^T  v_0|^2-1)\, d_*x \nonumber\\
  &=&\varkappa_j\intl_{S^n} Y_j (x) R^{(\rho, \sigma)}_{j/2} (2|x^T  v|^2-1)\, d_*x. \nonumber \eea
Taking into account that the sum in (\ref{tag 4.3}) has only one term, we obtain

\[ (M_\tau Y_j) (v) =\varkappa_j^2  R^{(\rho,\sigma)}_{j/2} (2\tau^2 -1) \,\intl_{S^n} Y_j (x) R^{(\rho, \sigma)}_{j/2} (2|x^T v|^2-1)\, d_*x.\]
 For technical reasons, it is convenient to set
\be\label {tag 4.7} \hat Y_j (v)= \a_j\intl_{S^n} Y_j (x) R^{(\rho, \sigma)}_{j/2} (2|x^T   v|^2-1)\, d_*x, \quad \a_j=\varkappa_j \sqrt {d_n(j)} \ee
(the role of the coefficient $\a_j$ will be clarified later) and
\be\label {tag 4.7r} \hat m_\tau (j)= \frac{\varkappa_j} {\sqrt {d_n(j)}} \, R^{(\rho,\sigma)}_{j/2} (2\tau^2 -1).\ee

This gives the following statement.

\begin{theorem}\label{Theorem 4.1a}  Let $v\in V_{n+1, n-k}$, $1\le k \le n-1$.
For any spherical harmonic $Y_j (x)$ of degree $j$ on $S^n$ and any $\tau \in [0,1]$,
\be\label {tag 4.11m}
(M_\tau Y_j) (v)= \left\{ \begin{array} {ll} 0  &\hbox{ if} \; \; j \; \;\hbox{  is odd}, \\
 \hat m_\tau (j) \hat Y_j (v) \; \; &\hbox{if} \; \; j \; \; \hbox{ is even}.\\
\end{array}
\right.\ee
\end{theorem}

An analogue of (\ref{tag 4.11m}) for the operator (\ref{3.1}) follows from (\ref{3.11}).
\begin{theorem}\label{Theorem 4.1}  Let $x\in S^n$, $v\in V_{n+1, n-k}$, $1\le k \le n-1$. If
\be\label {cda} \intl^1_0\!   |a(\tau)|\, \rho (\tau)\, d\tau <\infty,  \qquad  \rho (\tau) = (1-\tau^2)^{(k-1)/2} \tau^{n-k-1},\ee
then for any spherical harmonic $Y_j (x)$ of degree $j$,
\be\label {tag 4.11} (AY_j) (v)\equiv \intl_{S^n} a(|x^T   v|) Y_j(x) d_*x =
\left\{ \begin{array} {ll} 0  &\!\!\hbox{ if} \;  j \hbox{  is odd}, \\
 \hat a(j) \hat Y_j (v) \; \; &\!\!\hbox{if} \;  j\hbox{ is even}, \\
\end{array}
\right.\ee
where
\be\label {tag 4.8} \hat a(j) \!= \!   c_{n, k,j}\intl^1_0\! \! R^{(\rho,\sigma)}_{j/2} (2\tau^2\! -\!1)  a(\tau)  \rho (\tau) d\tau,\quad
 c_{n, k,j}\! =\!  \frac{\varkappa_j\, \sigma_{n-k-1} \sigma_k } {\sigma_n\,\sqrt {d_n(j)}}.\ee
\end{theorem}
\begin{definition} \label {uyt} The function  $\hat Y_j (v)$,   generated by the spherical harmonic $Y_j (x)$ according to (\ref{tag 4.7}), will be called
 the {\it induced Stiefel harmonic}. Because $\hat Y_j (v)$ is right $O(n-k)$-invariant, it can be regarded as a harmonic on the space $\Xi$ of $k$-geodesics  or on the  Grassmannian $G_{n+1, k+1}$.
Following standard terminology in harmonic analysis, we call $\hat a(j)$ and  $\hat m_\tau (j)$ the {\it multipliers} of the respective operators  $A$ and $M_\tau$.
\end{definition}

As we shall see below,  the induced Stiefel harmonics have a number of remarkable properties, similar to those of usual spherical harmonics.

A routine calculation shows that in the case $k=n-1$, Theorems \ref{Theorem 4.1a} and \ref{Theorem 4.1}  agree with the
corresponding statements for spherical harmonics in \cite[Appendix]{Ru15}. In particular, Theorem  \ref{Theorem 4.1} resembles the
classical Funk-Hecke theorem; cf. \cite[Theorem A.34]{Ru15}.

\subsection{Properties of the induced Stiefel harmonics}
${}$

\noindent Let $ {\mathscr Y}=\{ Y^n_{j,\lam} (x) \}$ be an orthonormal basis of spherical harmonics in $L^2(S^n)$.  For $j$ even, we denote
\be\label {tag 4.17}  \hat Y_{j, \lam} (v) = \a_j \intl_{S^n}\! Y^n_{j, \lam} (x) R_{j/2}^{(\rho, \sigma )}(2|x^T v|^2 \!-\! 1) d_*x, \quad \a_j=\varkappa_j \sqrt {d_n(j)},\ee
 where, as in (\ref{tag 4.7}), $ \rho = (k-1)/2$, $ \sigma = (n-k)/2 -1$.

  \begin{lemma}\label{Lemma 4.3}  Let $j\in \{0,2,4,\ldots \}$. The following addition formula holds:
 \be\label {tag 4.18} R^{(\rho,\sigma)}_{j/2} (2|x^T  v|^2-1) = \a_j^{-1} \suml^{d_n(j)}_{\lam = 1} Y^n_{j, \lam} (x) \hat Y_{j, \lam} (v). \ee
\end{lemma}

\begin{proof} Let us write (\ref{tag 4.18}) as $I(x)=J(x)$, assuming $v$ fixed.  The statement will be proved if we establish the coincidence of the Fourier-Laplace coefficients of $I(x)$ and $J(x)$, that is,
\[
\intl_{S^n} Y^n_{i, \mu} (x) I(x) d_*x= \intl_{S^n} Y^n_{i, \mu} (x) J(x) d_*x\]
 for any harmonic  $Y^n_{i, \mu} \in {\mathscr Y}$, not necessarily even.
Let us show that
  \be\label {tag 4.19} l.h.s.\equiv  \intl_{S^n} Y^n_{i, \mu} (x) I(x) d_*x = \left\{ \begin{array} {ll} \a_j^{-1}\, \hat Y_{j, \mu} (v) &\hbox{\rm if} \; \; i = j,\\
   0 &\; \; \hbox{\rm if} \; \; i \neq j.\\
\end{array}
\right.\ee
If $i = j$, this equality holds by definition (\ref{tag 4.17}) (set $\mu=\lam$). If  $i \neq j$ we write out (\ref{tag 4.11}) with $j$ replaced by $i$, that is,
  \be\label {tag 4.11new} (AY_i) (v)\equiv \intl_{S^n} a(|x^T  v|) Y_i(x) d_*x =
\left\{ \begin{array} {ll} 0  &\!\!\hbox{ if} \;  i \hbox{  is odd}, \\
 \hat a(i) \hat Y_i (v) \; \; &\!\!\hbox{if} \;  i\hbox{ is even}, \\
\end{array}
\right.\ee
 and then set
 $a(|x^T v|) = R^{(\rho,\sigma)}_{j/2} (2|x^T  v|^2-1)$, $Y_i =Y^n_{i, \mu}$. This gives
\[
 l.h.s.\equiv  \intl_{S^n} Y^n_{i, \mu} (x) I(x) dx= \hat a(i) Y^n_{i, \mu} (v)=0\]
  because, by (\ref{tag 4.8}),
 \bea \hat a(i) &=&  c_{n, k,i} \intl^1_0\!  (1-\tau^2)^{(k-1)/2}\,\tau^{n-k-1} \, R^{(\rho,\sigma)}_{i/2} (2\tau^2 -1) \,R^{(\rho, \sigma)}_{j/2}(2\tau^2-1)\,  d\tau\nonumber\\
&=&  c_{n, k,i}  \intl^1_{-1} (1-s)^\rho (1+s)^\sigma R^{(\rho,\sigma)}_{i/2} (s) R^{(\rho, \sigma)}_{j/2}(s) ds = 0 \nonumber\eea
  due to orthogonality of Jacobi polynomials; cf. (\ref{ced}).

For the right-hand side we have
\[
 r.h.s.\equiv \intl_{S^n} Y^n_{i, \mu} (x) J(x) d_*x =   \a_j^{-1} \suml^{d_n(j)}_{\nu = 1} \hat Y_{j, \lam} (v)   \intl_{S^n} Y^n_{i, \mu} (x) Y^n_{j, \lam} (x)) d_*x =0.\]
This completes the proof.
\end{proof}

Consider the set of all induced Stiefel harmonics (\ref{tag 4.17}) and recall the notation
$ \hat {\mathscr Y}=\{\hat Y_{j, \lam}:  \; j=0, 2,\ldots ;\quad \lam =1,2, \ldots d_n(j) \}$,
$d_n(j)$ being defined by (\ref{2.8}).

\begin{lemma}\label{Lemma 4.4}
The set
$ \hat {\mathscr Y}$ is orthonormal, that is,
 \be\label {tag 4.20}
 \intl_{V_{n+1, n-k}} \!\!\!\hat Y_{j, \lam} (v) \hat Y_{j', \lam'} (v) d_*v  =\left\{ \begin{array} {ll} 1
 &\;\;\hbox{\it if} \; \; j = j' \; \; \hbox{\it and} \; \; \lam = \lam',\\ 0 &\; \; \hbox{\it otherwise}.\\
\end{array}
\right.\ee
\end{lemma}
  \begin{proof}  We denote by $I$ the left-hand side of (\ref{tag 4.20}) and make use of (\ref{tag 4.17}).  Changing the order of integration, we obtain
    \be\label {tag 4.21} I = \intl_{S^n} Y^n_{j, \lam} (x) d_*x \intl_{S^n} Y^n_{j', \lam'} (y) \Omega_{j, j'} (x, y) d_*y, \ee
  \be\label {tag 4.22} \Omega_{j, j'} (x, y) \!=\! \a_j \a_{j'} \!\!\! \intl_{V_{n+1,n-k}} \!\!\!R^{(\rho, \sigma)}_{j/2} (2|x^T  v|^2 \!- \!1) R^{(\rho, \sigma)}_{j'/2} (2|y^T  v|^2 \!- \!1) d_*v.\ee
     By the rotation invariance, $\Omega_{j, j'} (x, y)$ is  a single-variable function of $x\cdot y$.
    Abusing notation, we set $ \Omega_{j, j'} (x, y) \equiv \Omega_{j, j'} (x\cdot y)$.  Then the Funk-Hecke formula (\ref{FH}) yields

        $$ I=\om_{j, j'} \intl_{S^n} Y_{j, \lam}(x) Y_{j', \lam'} (x) d_*x=\om_{j, j'}\, \del_{jj'}\,\del_{\lam \lam'},$$
 \be\label {tag 4.23} \om_{j, j'} = \frac{\sig_{n-1}}{\sig_n}   \intl^1_{-1} P_{j'} (t) \Omega_{j, j'}(t) (1-t^2)^{n/2-1} dt.\ee
 To complete the proof, we need to justify the equality
 \be\label {twsag} \om_{j, j} = 1.\ee
   First, let us show that
  \be\label {tag 4.24}
\Omega_{j, j} (t)= \frac{\a_j^2}{\varkappa_j^{2}}\, P_j (t).\ee
    By (\ref{tag 4.22})  and (\ref{tag 4.18}),
     \be\label {tag 4.25} \Omega_{j, j} (x\cdot y) = \a_j\suml^{d_n(j)}_{\lam = 1} Y^n_{j, \lam} (x) \intl_{V_{n+1, n-k}} \hat Y_{j, \lam} (v) R^{(\rho, \sigma)}_{j/2} (2|y^T  v|^2 - 1)d_*v.\ee
 It follows that $\Omega_{j, j} (x\cdot y)$ is a spherical harmonic of degree $j$ in the $x$-variable for each $y \in S^n$, and therefore,
 $\Omega_{j, j} (t)$ is a constant multiple of the spherical polynomial, i.e., $\Omega_{j, j} (t)= A_j P_j (t)$  $A_j =\const$; see, e.g., \cite [Lemma A.26]{Ru15}.

 To find $A_j$, we set $x = y = e_{n+1}$ and $  i = j$ in (\ref{tag 4.22}). This gives
 \[
 \Omega_{j, j} (1)= A_j P_j (1)=\a_j^2 \intl_{V_{n+1, n-k}} [R^{(\rho,\sigma)}_{j/2} (2|e_{n+1}^T  v|^2-1)]^2 d_*v.\]
 Taking into account that  $P_j (1)=1$ and using (\ref{3.12}) with $\vp=1$ and $ x=e_{n+1}$, we obtain
 \[ A_j =\a_j^2\, c_{n,k} \intl^1_0 (1\!-\!\tau^2)^{(k-1)/2} \tau^{n-k-1}[R^{(\rho,\sigma)}_{j/2} (2\tau^2\!-\!1)]^2  (M^*_\tau 1)(e_{n+1})\, d\tau,\]
where $(M^*_\tau 1)(e_{n+1}) = 1$, $c_{n,k}=\sigma_{n-k-1} \sigma_k/\sig_n$.
  Furthermore, from (\ref{2.11}) and (\ref{tag 4.9}) we have
\bea &&\intl^1_0 (1-\tau^2)^{(k-1)/2} \tau^{n-k-1}[R^{(\rho,\sigma)}_{j/2} (2\tau^2-1)]^2 \, d\tau\nonumber\\
&&=\frac{\Gamma^2 \left(\displaystyle { k+1\over 2}\right)\; \Gamma  \left(\displaystyle { j\over 2} + 1\right) \; \Gamma \left(\displaystyle { j+n-k\over 2}\right)}
{(2j+n-1) \; \Gamma \left(\displaystyle { j+k+1\over 2}\right)\; \Gamma  \left(\displaystyle { j+n-1\over 2}\right)}=   \frac{1}{\varkappa_j^{2}\, c_{n,k}}.\nonumber\eea
 Hence $A_j= \a_j^2/\varkappa_j^{2}$,  and  (\ref{tag 4.24}) follows.

Now (\ref{tag 4.23}) and (\ref{shap31})   yield
\[ \om_{j, j} = \frac{\sig_{n-1}\, \a_j^2}{\sig_n\, \varkappa_j^{2}} \,  \intl^1_{-1} [P_{j} (t)]^2  (1-t^2)^{n/2-1} dt=
\frac{ \a_j^2}{\varkappa_j^{2}\, d_n (j)}=1.\]
 This completes the proof.
  \end{proof}

 \begin{corollary}\label{Corollary 4.5}
 If the Stiefel harmonic $\hat Y_j $ is induced by the spherical harmonic $Y_j $ of even degree $j$, then
 \be\label {tag 4.27}  Y_{j} (x) = \a_j\!\! \intl_{V_{n+1, n-k}}\!\!\!\hat Y_{j} (v) R^{(\rho, \sigma)}_{j/2} (2|x^T  v|^2 \!-\! 1) d_*v, \quad \a_j=\varkappa_j \sqrt {d_n(j)}. \ee
 \end{corollary}
\begin{proof} By the addition formula (\ref{tag 4.18}), the right-hand side of (\ref{tag 4.27}) is a spherical harmonic of degree $j$. Hence
 it suffices to show that
\be\label {tag 4.27b}  Y^n_{j, \lam} (x) = \a_j \intl_{V_{n+1, n-k}}\hat Y_{j, \lam} (v) R^{(\rho, \sigma)}_{j/2} (2|x^T  v|^2 - 1) d_*v \ee
 for all basic  harmonics   $Y^n_{j,\lam} (x)$, $\;\lam \in \{1,2, \ldots, d_n(j)\}$. To this end,
  we represent $ R^{(\rho, \sigma)}_{j/2} (2|x^T  v|^2 - 1)$ by (\ref{tag 4.18}) and make use of the orthogonality (\ref{tag 4.20}). This gives
  \[ r.h.s.=  \suml^{d_n(j)}_{\lam' = 1} Y^n_{j, \lam'} (x)   \intl_{V_{n+1, n-k}}\!\!\!\hat Y_{j, \lam} (v) \hat Y_{j, \lam'} (v) d_*v= Y^n_{j, \lam} (x),\]
as desired.
  \end{proof}

\subsection{The Dual Statements}
${}$

\noindent The next Funk-Hecke type statement is  dual to Theorem \ref{Theorem 4.1}.

\begin{theorem}\label {Corollary 4.7} If the Stiefel harmonic $\hat Y_j $ is  induced by the spherical harmonic $Y_j $ of even degree $j$ and $a(\tau)$ satisfies (\ref{cda}),
then
\be\label {tag 4.28} \intl_{V_{n+1, n-k}} a(|x^T  v|)\, \hat Y_{j} (v)\, d_*v = \hat a(j)Y_{j} (x),\ee
$\hat a (j)$ being the multiplier (\ref{tag 4.8}).
\end{theorem}
\begin{proof}  As above, it suffices to show that
\be\label {tag 4.28y} \intl_{V_{n+1, n-k}} a(|x^T  v|) \hat Y_{j,\lam} (v) d_*v = \hat a(j)Y^n_{j, \lam} (x),\ee
for all basic harmonics   $Y^n_{j,\lam} (x)$, $\;\lam \in \{1,2, \ldots, d_n(j)\}$.
To prove (\ref{tag 4.28y}), we evaluate the Fourier-Laplace coefficients of both sides.  Let  $Y^n_{i, \mu}$ be an arbitrary spherical harmonic belonging to the orthonormal basis ${\mathscr Y}$ of $L^2 (S^n)$.
 Changing the order of integration, owing to (\ref{tag 4.11}) and (\ref{tag 4.20}),  we have
$$\intl_{S^n} Y^n_{i, \mu} (x) d_*x \intl_{V_{n+1, n-k}} a (|x' v|)\hat Y_{j, \lam} (v) d_*v = \hat a (i) \delta_{ij} \delta_{\mu\lam}. $$
Since the right-hand side has the same Fourier-Laplace coefficients, the result follows.
\end{proof}

The following theorem is dual to Theorem \ref {Theorem 4.1a}.
\begin{theorem} \label{gtr}  If the Stiefel harmonic $\hat Y_j (v)$ is  induced by the spherical harmonic $Y_j (x)$ of even degree $j$, then
\be\label {tagsq}
(M^*_\tau \hat Y_j)(x)=  \hat m_\tau(j) \,Y_j (x),\ee
where $\hat m_\tau(j)$ is the multiplier (\ref{tag 4.7r}).
\end{theorem}
\begin{proof}
It suffices to prove (\ref{tagsq}) on basic  harmonics, that is,
\be\label {tagsq1}
(M^*_\tau \hat Y_{j,\lam})(x)= \hat m_\tau(j)\, Y_{j,\lam} (x),\qquad Y^n_{j,\lam} \in {\mathscr Y}.  \ee
 For any spherical harmonic $Y^n_{j',\lam'} \in {\mathscr Y} $, owing to (\ref{3.10}), (\ref{tag 4.11m}) and (\ref {Lemma 4.4}), we have
\bea
&& \intl_{S^n} Y^n_{j',\lam'}(x) (M^*_\tau \hat Y_{j,\lam})(x)\, d_*x=\intl_{V_{n+1, n-k}} \hat Y_{j,\lam} (v) (M_\tau  Y_{j',\lam'})(v)\,  d_*v \nonumber\\
&&=\hat m_\tau(j) \intl_{V_{n+1, n-k}} \hat Y_{j,\lam} (v)\hat Y_{j',\lam'} (v)\, d_*v = \hat m_\tau(j) \delta_{jj'} \delta_{\lam \lam'}.\nonumber\eea
Similarly, for the right-hand side of (\ref{tagsq1})   we have
\[
\hat m_\tau(j) \intl_{S^n} Y^n_{j',\lam'}(x) Y^n_{j,\lam} (x) d_*x=\hat m_\tau(j) \delta_{jj'} \delta_{\lam \lam'}.\]
Thus all the Fourier-Laplace coefficients of the both sides of (\ref{tagsq1}) coincide, and the result follows.
\end{proof}

\subsection{Examples}\label{Examples}

\begin{example}\label {Example 4.9.}  Consider the Funk-Radon transforms
\[(Rf)(S^n \cap v^\perp)= (M_0 f)(v), \qquad  (R^* \vp)(x) = (M^*_0 \vp)(x), \]
where  $v\in V_{n+1, n-k}$, $x\in S^n$; cf. (\ref {cdfae}).
 Setting $\tau=0$ in (\ref{tag 4.11m}) and (\ref{tagsq}), we compute the multiplier  $\hat m_j = \hat m_0 (j)$ of these  operators. Specifically, by (\ref{tag 4.7r}),
 \[ \hat m_j = \frac{\varkappa_j} {\sqrt {d_n(j)}} \, R^{(\rho,\sigma)}_{j/2} (-1).\]
Here, by (\ref{2.10}),
\bea
R^{(\rho,\sigma)}_{j/2} (-1)&=& \frac{P^{(\rho, \sigma)}_{j/2} (-1)}{P_{j/2}^{(\rho,\sigma)} (1)}= \frac{ (-1)^{j/2} P^{(\sigma, \rho)}_{j/2} (1)}{P_{j/2}^{(\rho,\sigma)} (1)}\nonumber\\
&=& (-1)^{j/2} \frac{\Gamma \left(\displaystyle {k+1\over 2}\right)\; \Gamma \left(\displaystyle { j+n-k\over 2}\right)}
{ \Gamma \left(\displaystyle { n-k\over 2}\right) \;\Gamma \left(\displaystyle { j+k+1\over 2}\right)\; }.\nonumber\eea
Hence, by (\ref{tag 4.9}) and (\ref{2.8}), a simple calculation yields
\be\label {tedm2} \hat m_j = (-1)^j \del_{n,k}  \left [ \frac{\Gamma \left(\displaystyle {j+n-k\over 2}\right)\; \Gamma  \left(\displaystyle { j+1\over 2}\right)}
{\Gamma  \left(\displaystyle { j+n\over 2} \right) \; \Gamma \left(\displaystyle { j+k +1\over 2}\right)}\right ]^{1/2}\!\!\!,\qquad \quad\ee
\[
\del_{n,k}= \left [\frac{\Gamma \left(\displaystyle {k +1 \over 2}\right)\; \Gamma  \left(\displaystyle { n\over 2}\right)}
{\Gamma  \left(\displaystyle {n -k\over 2} \right) \; \pi^{1/2}}\right ]^{1/2}.\]

In particular, if $k=n-1$, then
\[
\hat m_j = (-1)^{j/2} \frac {\Gamma  \left(\displaystyle { n\over 2}\right)\, \Gamma  \left(\displaystyle { j+1\over 2}\right)}{ \pi^{1/2} \Gamma  \left(\displaystyle { j+n\over 2}\right)}.\]
This expression agrees with the known Fourier-Laplace multiplier of the Funk transform on $S^{n-1}$; cf.  \cite[formula (5.1.3)]{Ru15}.
\end{example}

\begin{example}\label {Example 4.9hh.}  The intertwining operator (\ref{3.1}) with the kernel $a(t) = t^{\a-n+k}, \;  Re \,\a> 0$, is a constant
multiple of the generalized cosine transform in integral geometry \cite{Ru02b, Ru08}. Let
\bea \label {3.1f} (A_\a f)(v) &=& \intl_{S^n} |x^T  v|^{\a-n+k} f(x) \,d_*x, \qquad v \in V_{n+1, n-k},\\
\label {3.2f}({A_\a ^*}\vp)(x) &=& \intl_{V_{n+1, n-k}} \!\!\! |x^T  v|^{\a-n+k}  \vp (v) \,d_*v, \qquad x \in S^n.\eea

By (\ref{tag 4.8}), the multiplier of these operators  is
\[ \hat a_\a(j) =
c_{n, k,j}\intl^1_0\!  (1-\tau^2)^{(k-1)/2}\,\tau^{\a-1} \, R^{(\rho,\sigma)}_{j/2} (2\tau^2 -1) \, d\tau.\]
This integral can be  evaluated using  \cite[2.22.2(9)]{PBM} and the properties of Jacobi polynomials (we skip the routine calculations). The result is
\be\label {tag 4.29} \hat a_\a(j) = (-1)^{j/2} \del_{\a,n,k} \,\hat m_\a (j),\ee
where
 \[ \del_{\a,n,k}= \left [\frac{\sig_{n-k-1} \sig_k \,\Gam (n)}{2^n \sig_n}\right ]^{1/2} \frac{\Gamma (\a/2)}{\Gamma \left((n-k-\a)/2\right)},\]
\[\hat m_\a (j)=\left [ \frac{\Gamma\left(\displaystyle{ j+k+1 \over 2} \right) \, \Gamma \left(\displaystyle{j+1\over 2} \right)}
 {\Gamma \left(\displaystyle{j+n-k\over 2}\right) \, \Gamma \left(\displaystyle {n+j\over 2} \right)} \right]^{1/2}
 \, \frac {\Gamma\left(\displaystyle{ j+n-k-\a \over 2} \right)}{\Gamma\left(\displaystyle{ j+k+\a +1\over 2} \right)}.\]

In the case $k=n-1$, this expression agrees (up to notation) with the known Fourier-Laplace multiplier of the $\a$-cosine transform on $S^{n-1}$; cf.  \cite[formula (5.1.9)]{Ru15}.
If $j \to \infty$, then $\hat a_\a(j)=O(j^{-k/2 -\a})$.
\end{example}

\subsection{Proof of Theorem \ref{inje}}

${}$

\noindent {\rm (i)} Let $\tau =\sin t$. By (\ref{cdf}), it suffices to prove injectivity of the mapping
\[L^1_{even} (S^n) \ni f \longrightarrow M_\tau f \in L^1 (V_{n+1, n-k}).\]
Denote $\mathcal{P}_j(\tau)=R^{(\rho,\sigma)}_{j/2} (2\tau^2-1)$ and suppose that $M_\tau$ is injective, i.e.,  $M_\tau f = 0$ implies $f =0$ a.e. for every  $f\in L^1_{even}(S^n)$. Assuming the contrary,
that is, $\mathcal{P}_j(\tau)= 0$ for some $j=j_0 \in \{0,2,4,\ldots \}$, we obtain  that $M_\tau Y_{j_0}=c \,  \mathcal{P}_{j_0}(\tau) \hat Y_{j_0} \equiv 0$ for
every spherical harmonic $Y_{j_0}$ on $S^n$. Hence, by the injectivity assumption,  $Y_{j_0} \equiv 0$, which gives a contradiction.

Conversely, suppose that $\mathcal{P}_j(\tau)\neq 0$ for all $j\in \{0,2,4, \ldots\}$, and let  $M_\tau f = 0$ for some $f\in L^1_{even}(S^n)$. Then
for every spherical harmonic $Y_j$ on $S^n$ with $j$ even and the corresponding Stiefel harmonic $\hat Y_j$, owing to (\ref{3.10})   and (\ref{tagsq}), we obtain
\bea 0&=&\intl_{V_{n+1, n-k}} (M_\tau f)(v) \hat Y_j (v) d_*v = \intl_{S^n} f(x) (M^*_\tau \hat Y_j) (x)  d_*x \nonumber\\
\label {kaq}&=&c\, \mathcal{P}_j(\tau) \intl_{S^n} f(x) Y_j (x) d_*x.\eea
 Because $\mathcal{P}_j(\tau)\neq 0$, it follows that all the Fourier-Laplace coefficients of $f$ are zero. Hence $f(x)=0$ for almost all $x\in S^n$ (use. e.g., \cite[Proposition A.18]{Ru15}).
 Now the statement of Theorem \ref{inje} follows if we set
\be\label {tre} R^{(\rho,\sigma)}_{j/2} (2\tau^2-1) = R^{(\rho,\sigma)}_{j/2} (2\sin^2 t-1)= (-1)^{j/2}  R^{(\sigma,\rho)}_{j/2} (\cos\, 2t).\ee

${}$
\noindent {\rm (ii)}  More generally,   let $\tau_i =\sin t_i$. Then $R_{t_i} f =M_{\tau_i} f$ and the system ({\ref{sol}) is equivalent to
 \be\label {solp}
 \mathcal{P}_j(\tau_i)= 0, \qquad  i=1,2,\ldots, \ell. \ee
 If these equations have no common  solution in the $j$-variable, then for any $j\in \{0,2,4,\ldots \}$  there exists
  at least one $i=i (j)$ such that
  $\mathcal{P}_j(\tau_{i(j)})\neq 0$. If $M_{\tau_i} f=0$ for all $i=1,2,\ldots, \ell$, then, in particular, $M_{\tau_{i(j)}} f=0$ and, as in (\ref{kaq}),
  \[ \mathcal{P}_j(\tau_{i(j)}) \intl_{S^n} f(x) Y_j (x) d_*x =0.\]
 This implies $\int_{S^n} f(x) Y_j (x) d_*x =0$. Because $j\in \{0,2,4,\ldots \}$ is arbitrary, it follows that $f=0$ a.e. on $S^n$.

  Conversely, if the equations (\ref {solp}) have a common solution, say, $j=j_0$, then, as above, for any spherical harmonic $Y_{j_0}$ and all $i=1,2,\ldots, \ell$, we have
  \[ M_{\tau_i} Y_{j_0}=c \,  \mathcal{P}_{j_0}(\tau_i) \hat Y_{j_0} \equiv 0.\]
  This completes the proof  if we take into account (\ref {tre}).
 \hfill $\Box$.

\section {Conclusion}\label {Conclusion}

Some comments are in order.

\noindent {\bf 1.} The purpose of the paper was two-fold.
On the one hand, it would be interesting to investigate  injectivity of the shifted Radon  transforms on  an arbitrary constant curvature space $X$. This setting of the problem extends the well known consideration of spherical means with center at a point to the case of a `multidimensional center'.
To start with,   we restricted to the  case $X=S^n$ and  obtained necessary and sufficient conditions of injectivity of the shifted Funk-Radon transform on $L^1(S^n)$.  The cases, when $X$ is the Euclidean or hyperbolic space,  are left for the future.

On the other hand, our study needs a suitable harmonic analysis, which makes a bridge between functions on the sphere and functions on the Stiefel (or Grassmann) manifolds.
  This analysis is of independent interest and has many  aspects.  We considered only some of them, which are related to
 the  induced  orthonormal systems,  the corresponding Funk-Hecke type theorems, the addition formula, and multipliers.

 It is natural to
 conjecture that our consideration paves the way to further investigations. The corresponding  theory for translation invariant linear operators in $\bbr^n$ is well known; see, e.g., H\"{o}rmander \cite {Ho},   Stein and   Weiss \cite{SW}, Grafakos \cite{Graf}. Operators on the unit sphere commuting with rotations  were studied by  Coifman and  Weiss \cite{CoW},  Dunkl \cite{Du},
 Rubin \cite[Sections A.10 - A.13]{Ru15},  Samko \cite{Sa}, to mention a few.

\vskip 0.2 truecm

\noindent {\bf 2.}   Our formula (\ref {tedm2}) for the multiplier of the Funk-Radon transform differs from that suggested by Strichartz \cite{Str1, Str3}. Nevertheless, multipliers in (\ref {tedm2}) and in \cite{Str1} have the same order $O(j^{-k/2})$ as $j \to \infty$. Strichartz's approach  relies on his previous group-theoretic considerations in \cite{Str2};
cf. formulas (4.2) and (4.4) in  \cite{Str1}, where  the multipliers of the Funk-Radon transform and its dual have different analytic expression.
Our approach is essentially different, self-contained,  and  invokes Jacobi polynomials. It is applicable to more general intertwining operators, yields the relevant
Funk-Hecke type theorems, and the addition formula for the corresponding harmonics. Moreover, unlike \cite{Str1}, our  multipliers for the  intertwining operators and the dual ones are the same.

\vskip 0.2 truecm

\noindent {\bf 3.}  According to Strichartz \cite[Theorem 4.2] {Str3},
 the asymptotics  $O(j^{-k/2})$ of the  multiplier corresponding to the Funk-Radon transform $R$ and combined with the oscillating factor $(-1)^{j/2}$,  yields $L^p$-$L^q$ estimates of $R$.
 However, there is an intimate connection   between $R$ and the $k$-plane transform on $\bbr^n$  (see, e.g., \cite[Section 3]{Ru19}), which allows one to convert boundedness results for one class of operators to the similar results for another.  This conversion is performed with preservation of the operator norms and yields important geometric inequalities for sections of convex bodies in integral geometry.

  It might be of interest to convert  Strichartz's estimates from \cite[Theorem 4.2] {Str3} to those for the $k$-plane transforms and  compare the obtained statements with known results  by
  Christ \cite{Chr84} and Drury \cite{Dr89}.   We conjecture that this approach is applicable to  more general analytic family of operators (\ref{3.1f}) with the
   oscillatory multiplier (\ref{tag 4.29}), having the order   $O(j^{-k/2 -\a})$.

\vskip 0.2 truecm

\noindent {\bf 4.} Natural higher-rank generalizations of the spherical means and the corresponding shifted Radon transforms
arise in  integral geometry on Grassmann manifolds and matrix spaces \cite {GR, OR08}. In this setting,  an analogue of the shift $t$ is
matrix-valued and represented by  a positive definite  matrix.  It might be of interest to study injectivity of such higher-rank  mean value operators and develop the relevant harmonic analysis.

\vskip 0.2 truecm

\noindent {\bf Acknowledgement.} The author is grateful to Professor Mark Agranovsky for helpful inspiring discussions and sharing his knowledge of the subject.

\end{document}